\documentclass{birkjour}
\newtheorem{theorem}{Theorem}[section]
 
 \newtheorem{lemma}[theorem]{Lemma}
 
 \theoremstyle{definition}
 
 \theoremstyle{remark}
 \newtheorem{remark}[theorem]{Remark}
 \newtheorem*{example}{Example}
 \numberwithin{equation}{section}

\begin{document}
\title{Solutions of a class of nonlinear matrix equations}

\author[Samik Pakhira]{Samik Pakhira}
\address{Department of Mathematics \& Statistics, Aliah University, II A/27, New Town, Kolkata-160, West Bengal, India}
\email{samikpakhira@gmail.com}

\author[Snehasish Bose]{Snehasish Bose}
\address{Theoretical Statistics and Mathematics Unit, ISI Bangalore, 8th Mile Mysore Road, Bangalore 560059, India}
\email{bosesonay@gmail.com}

\author[Sk Monowar Hossein]{Sk Monowar Hossein}
\address{Department of Mathematics \& Statistics, Aliah University, II A/27, New Town, Kolkata-160, West Bengal, India}
\email{skmonowar.math@aliah.ac.in}


\begin{abstract}
	In this article we present several necessary and sufficient conditions for the existence of Hermitian positive definite solutions of nonlinear matrix equations of the form 
	$X^s+A^*X^{-t}A+B^*X^{-p}B=Q$, where  $s,t,p \geq 1$, $A,B$ are nonsingular matrices and $Q$ is a Hermitian positive definite matrix. We derive some iterations to compute the solutions followed by some examples. In this context we also discuss about the maximal and the minimal Hermitian positive definite solution of this particular nonlinear matrix equation. 
\end{abstract}

\maketitle
\section{Introduction and Preliminaries}
Let $GL(n)$ be the set of all $n\times n$ nonsingular matrices and $\mathcal{P}(n)$ be the set of all Hermitian positive definite matrices.
Consider the nonlinear matrix equations of the form 
\begin{equation}\label{eqi:1}
X^s+A^*X^{-t}A+B^*X^{-p}B=Q
\end{equation}
where  $s,t,p \geq 1$, $A,B \in GL(n)$ and $Q \in \mathcal{P}(n)$. There are several problems in control theory, dynamical programming, ladder networks, statistics, etc., where this type of equations play an especially significant role. Thus it has always been a major concern to derive proficient techniques to solve nonlinear matrix equations of the type of \eqref{eqi:1}. Over the years several mathematicians solved this group of nonlinear matrix equations using various types of methods.
 To name a few: Anderson et al. \cite{AMT1990} ($X=A-BX^{-1}B^*$), Engwarda et al. \cite{ERR1993} ($X + A^*X^{-1}A = I$), Ferrante and Levy \cite{FL1996} ($X=Q-NX^{-1}N^*$), Meini \cite{M2002} ($X\pm A^*X^{-1}A=Q$), Ivanov \cite{I2006} ($X + A^*X^{-n}A = Q$), Long et al. \cite{LHZ2008} ($X+A^*X^{-1}A+B^*X^{-1}B=I$), Popchev \cite{PPKA2011} ($X+A^*X^{-1}A+B^*X^{-1}B=I$), Liu and Chen \cite{LC2011} ($X^s+A^*X^{-t_1}A+B^*X^{-t_2}B=Q$,~ $s\geq 1, 0<t_1,t_2\leq 1$), Vaezzadeh et al. \cite{VVSP2013} ($X+A^*X^{-1}A+B^*X^{-1}B=Q$), Hasanov \cite{H2017} ($X+A^*X^{-1}A-B^*X^{-1}B=I$), Hasanov et al. \cite{HA2015} ($X + A^*X^{-1}A + B^*X^{-1}B = Q$) and many more.
 
 Fixed point theory is one of the techniques that plays a definitive role for computing solutions of various types of nonlinear matrix equations. Ran and Reurings \cite{RR2004} using Ky Fan norm and an analogous result of Banach fixed point principle solved nonlinear matrix equations of the form $X=Q\pm \sum_{i=1}^m{A_i}^*F(X){A_i}$. In this context they proposed the following result.
 \begin{theorem}  \label{lemma 7} (\cite{RR2004})
 	Let $\mathbb{X}$ be a partially ordered set such that every pair $x,y\in \mathbb{X}$ has a lower bound and an upper bound. Furthermore, let $d$ be a metric on $\mathbb{X}$ such that $(\mathbb{X},d)$ be a complete metric space. If $\mathcal{F}$ is a continuous, monotone (order-preserving or order-reversing) map from $\mathbb{X}$ into $\mathbb{X}$ such that
 	\begin{enumerate}
 		\item[(1)] there exists $0<c<1~:~d(\mathcal{F}(x),\mathcal{F}(y))\leq c d(x,y)$, for all $x\geq y$,
 		\item[(2)] there exists $x_0\in\mathbb{X}~:~x_0\leq \mathcal{F}(x_0)$ or $x_0\geq \mathcal{F}(x_0)$,
 	\end{enumerate}
 	then $\mathcal{F}$ has a fixed point $\bar{x}$. Moreover, for every $x\in\mathbb{X}$,
 	\begin{equation*}
 	\lim_{n\to\infty}\mathcal{F}^n(x)=\bar{x}.
 	\end{equation*}
 \end{theorem}
Using this type of contraction theorems many authors solved different types of nonlinear matrix equations, such as, Lim \cite{L2009}, Duan and Liao \cite{DL2009a}, Berzig and Samet \cite{BS2011}, etc. Later Bose et al. \cite{SMK2016} generalizes their results by introducing a notion of $\omega$-distance in partially ordered $G$-metric spaces and solved the nonlinear matrix equation $X=Q\pm \sum_{i=1}^m{A_i}^*F(X){A_i}$. \\
Let $(\mathbb{X},\leq)$ be a partially ordered set and $d$ be a metric on $\mathbb{X}$. A mapping $F:\mathbb{X}\times\mathbb{X}\rightarrow\mathbb{X}$ has a mixed monotone property \cite{BL2006}
if $F(x,y)$ is monotone non-decreasing in $x$ and is monotone non-increasing in $y$, i.e, for any $x,y\in\mathbb{X}$,
\begin{equation*}
x_1,x_2\in\mathbb{X},~~x_1\leq{x_2}~\Rightarrow~F(x_1,y)\leq{F(x_2,y)}
\end{equation*}
and
\begin{equation*}
y_1,y_2\in\mathbb{X},~~y_1\leq{y_2}~\Rightarrow~F(x,y_1)\geq{F(x,y_2)}.
\end{equation*}
\newline
A pair $(x,y)\in\mathbb{X}\times\mathbb{X}$ is called a coupled fixed point 
of a mapping $F:\mathbb{X}\times\mathbb{X}\rightarrow\mathbb{X}$ if $F(x,y)=x$ and $F(y,x)=y$.\\
In \cite{BL2006} Bhaskar and Lakshmikantham presented the following coupled fixed point theorem using the mixed monotone property in partial ordered metric spaces and applied it to solve periodic boundary value problem.
\begin{theorem}  \label{lemma 5}(\cite{BL2006}) 
	Let $(\mathbb{X},\leq)$ be a partially ordered set and $d$ be a metric on $X$. Let the map $F: \mathbb{X} \times \mathbb{X} \rightarrow \mathbb{X}$ be continuous and mixed monotone on $\mathbb{X}$. Assume that there exists a $\delta \in [0,1)$ with 
	\begin{equation} \label{eq:2}
	d(F(x,y),~F(u,v)) \leq \frac{\delta}{2} [d(x,u)+d(y,v)],
	\end{equation}
	for all $x \geq u$ and $y \leq v$. Suppose also that 
	\begin{enumerate}
		\item[(i)]  there exist $x_0,~y_0 \in \mathbb{X}$ such that $x_0 \leq F(x_0,y_0)$ and $y_0 \geq F(y_0,x_0)$;
		\item[(ii)]  every pair of elements has either a lower bound or an upper bound.
	\end{enumerate}
	Then there exists a unique $\bar{x} \in \mathbb{X}$ such that $\bar{x}=F(\bar{x},\bar{x})$. Moreover, the sequences $\{x_k\}$ and $\{y_k\}$ generated by $x_{k+1}=F(x_k,y_k)$ and $y_{k+1}=F(y_k,x_k)$ converge to $\bar{x}$, with the following estimate
	\begin{equation*}
	\max \{d(x_k,\bar{x}),d(y_k,\bar{x})\} \leq \frac{\delta^k}{1-\delta} \max \{d(x_1,x_0), d(y_1,y_0)\}.
	\end{equation*}
\end{theorem}
Coupled fixed point theorem is one of the most heavily used tools to solve nonlinear matrix equations. In this setting Liu and Chen \cite{LC2011}, Berzig et al. \cite{BDS2012}, Hasanov \cite{H2017}, Asgari and Mousavi \cite{AM2015} and many more used this tool to compute the solutions of different groups of nonlinear matrix equations.

With the above discussion in mind, in this article, we consider nonlinear matrix equations of the form (\ref{eqi:1}). We present several necessary and sufficient conditions for the existence of a Hermitian positive definite solution of nonlinear matrix equation (\ref{eqi:1}). With the help of Theorem \ref{lemma 7}, Theorem \ref{lemma 5}, we derive some algorithms to compute the solutions. We also discuss about the maximal and the minimal Hermitian positive definite solution of (\ref{eqi:1}). Finally, to illustrate the scenarios, we provide with some examples.

Throughout this article we denote $\mathcal{H}(n) $ by the set of all $n\times{n}$ Hermitian positive definite matrices and denote $\mathcal{K}(n)$ by the set of all  $n\times{n}$ Hermitian positive semidefinite matrices. We write $A\geq{B}$ (or $A>B$) for $A,B\in\mathcal{H}(n)$ if $A-B\in\mathcal{K}(n)$ (or $A-B\in\mathcal{P}(n)$). In particular, we write $A\geq{0}$ (or $A>0$) if  $A\in\mathcal{K}(n)$ (or $A\in\mathcal{P}(n)$). We denote maximum eigenvalue of a matrix by $\lambda_1(.)$ and minimum eigenvalue by $\lambda_n(.)$. We use $\|.\|$ as spectral norm. We also denote spectral radius of a matrix $A$ by $\rho(A)$.   

Next we give some notable results which we use further in this article.

\begin{lemma}  \label{lemma 2} (\cite{F1998})
	Let $A$ and $B$ be positive operators on a Hilbert space $\mathcal{H}$ such that $M_1I \geq A \geq m_1I>0$, $M_2I \geq B \geq m_2I>0$ and $0< A \leq B$. Then for all $r \geq 1$
	\begin{equation*}
	A^r \leq \Big(\frac{M_1}{m_1}\Big)^{r-1}B^r, ~~A^r \leq \Big(\frac{M_2}{m_2}\Big)^{r-1}B^r.
	\end{equation*}
\end{lemma}
A norm $|||\cdot|||$ on $\mathcal{M}(n)$ is called unitarily invariant norm \cite{Bhatia} if 
\begin{equation}
|||UAV|||=|||A|||,
\end{equation}
for all $A$ and for all unitary matrices $U$ and $V$. Spectral norm is a unitarily invariant norm.

\begin{lemma}  \label{lemma 6}(\cite{Bhatia})
	Let $A$, $B$ and $C \in \mathcal{M}(n)$, set of all $n \times n$ matrices. Then for all unitarily invariant norm $||| \cdot~|||$,
	$|||BAC||| \leq \|B\|~|||A|||~\|C\|$.
\end{lemma}

\begin{lemma} \label{lemma 3} (\cite{Bhatia})
	If $0< \theta \leq 1$, and $P$ and $Q$ are Hermitian positive definite matrices of the same order with $P,~Q \geq bI > 0$, then for every unitarily invariant norm $|||P^\theta-Q^\theta||| \leq \theta b^{\theta -1}|||P-Q|||$ and $|||P^{-\theta}-Q^{-\theta}||| \leq \theta b^{-(\theta +1)} |||P-Q|||$. 
\end{lemma}

\begin{lemma} \label{lemma 4} (\cite{Bhatia})
	If $A \geq B > 0$ (or $A>B>0$), then $A^\alpha \geq B^\alpha>0$ (or $A^\alpha > B^\alpha>0$) for all $\alpha \in (0,1]$, and $B^\alpha \geq A^\alpha >0$ (or $B^\alpha > A^\alpha >0$) for all $\alpha \in [-1,0)$. And if $s \geq 1$ then $A \geq B \Rightarrow A^s \geq B^s$ if $A$ and $B$ commute.
\end{lemma}

\section{Main Results}
Consider the nonlinear matrix equation 
\begin{equation}\label{eq:1}
X^s+A^*X^{-t}A+B^*X^{-p}B=Q,
\end{equation}
where  $s,t,p \geq 1$, $A,B \in GL(n)$ and $Q \in \mathcal{P}(n)$.\\
Let $X^s=Y$. Then the equation (\ref{eq:1}) reduces to 
\begin{equation}\label{equ:2}
Y+A^*Y^{-\frac{t}{s}}A+B^*Y^{-\frac{p}{s}}B=Q,
\end{equation}
where $s,t,p \geq 1$, $A,B \in GL(n)$, $Q \in \mathcal{P}(n)$.
Therefore if $X$ is a Hermitian positive definite solution of (\ref{eq:1}), then $X^s$ is a Hermitian positive definite solution of (\ref{equ:2}). Also if $Y$ is a Hermitian positive definite solution of (\ref{equ:2}), then $Y^\frac{1}{s}$ is a Hermitian positive definite solution of (\ref{eq:1}). Thus we get the following theorem. 
\begin{theorem}
Equation (\ref{eq:1}) has a Hermitian positive definite solution if and only if equation (\ref{equ:2}) has a Hermitian positive definite solution.	
\end{theorem} 

Next we give some necessary conditions for the existence of Hermitian positive definite solution of (\ref{eq:1}). 
\begin{theorem}\label{theorem:1}
	Let $k={\lambda}_1(Q)\leq 1$. If equation (\ref{eq:1}) has a Hermitian positive definite solution then $\rho^2(A)< \frac{q^q}{(q+1)^{(q+1)}}$ and $\rho^2(B)< \frac{q^q}{(q+1)^{(q+1)}}$, where $q=\min\{\frac{t}{s}, \frac{p}{s}\}$.
\end{theorem}
\begin{proof}
Let $X$ be a Hermitian positive definite solution of (\ref{eq:1}). Then $Y=X^s$ is a Hermitian positive definite solution of (\ref{equ:2}). Therefore $Y=Q-A^*Y^{-\frac{t}{s}}A-B^*Y^{-\frac{p}{s}}B< Q$. Also since ${\lambda}_1(Q)\leq 1$, then $Y< Q\leq I$.\\
Let $\lambda_A$ be any eigenvalue of $A$, and $e_A$ be the corresponding unit eigenvector of $\lambda _A$. Multiplying from left in both side of (\ref{equ:2}) by $e_A^*$ and from right by $e_A$, we have
	\begin{equation*}
	\begin{split}
	&~~~~e_A^*Ye_A+e_A^*A^*Y^{-\frac{t}{s}}Ae_A+e_A^*B^*Y^{-\frac{p}{s}}Be_A=e_A^*Qe_A\\
	&\Rightarrow e_A^*Ye_A+|\lambda_A|^2e_A^*Y^{-\frac{t}{s}}e_A=e_A^*Qe_A-e_A^*B^*Y^{-\frac{p}{s}}Be_A\\
	&\Rightarrow e_A^*\big(Y+|\lambda_A|^2Y^{-\frac{t}{s}}\big)e_A< e_A^*Qe_A\leq I.
	\end{split}
	\end{equation*}
	Thus we have
	\begin{equation}\label{1}
	{\lambda}_n\big(Y+|\lambda_A|^2Y^{-\frac{t}{s}}\big)I\leq e_A^*\big(Y+|\lambda_A|^2Y^{-\frac{t}{s}}\big)e_A< I.
	\end{equation}
	Let  $q=\min\big\{\frac{t}{s}, \frac{p}{s}\big\}$ and $\lambda_i,~ i=1,..,n$ be the  eigenvalues of $Y$. Then for each $i=1,2,..,n$, $0<\lambda_i< 1$ and 
	\begin{equation*}
	\begin{split}
	{\lambda}_n\big(Y+|\lambda_A|^2Y^{-\frac{t}{s}}\big)&=\min_{1\leq i\leq n}\big(\lambda_i+|\lambda_A|^2{\lambda_i}^{-\frac{t}{s}}\big)\\
	&\geq \min_{1\leq i\leq n}\big(\lambda_i+|\lambda_A|^2{\lambda_i}^{-q}\big).
	\end{split}
	\end{equation*}
    Let $f(x)=x+|\lambda_A|^2x^{-q}$. Then $f'(x)=1-q|\lambda_A|^2x^{-(q+1)}$ and $f''(x)=q(q+1)|\lambda_A|^2x^{-q-2}$. Thus $f''(x)>0$, for $x>0$. Therefore, $\min f(x)=f\big((q|\lambda_A|^2)^\frac{1}{q+1}\big)=\frac{(q+1)|\lambda_A|^\frac{2}{q+1}}{q^\frac{q}{q+1}}$. As $0<\lambda_i< 1$, we have 
    \begin{equation*}
    \min_{1\leq i\leq n}\big(\lambda_i+|\lambda_A|^2{\lambda_i}^{-q}\big)\geq \frac{(q+1)|\lambda_A|^\frac{2}{q+1}}{q^\frac{q}{q+1}}.
    \end{equation*}
Thus, from (\ref{1}) we have 
\begin{equation*}
\begin{split}
&~~~\frac{(q+1)|\lambda_A|^\frac{2}{q+1}}{q^\frac{q}{q+1}}I\leq {\lambda}_n\big(Y+|\lambda_A|^2Y^{-\frac{t}{s}}\big)I< I\\
&\Rightarrow \frac{(q+1)|\lambda_A|^\frac{2}{q+1}}{q^\frac{q}{q+1}} < 1\\
&\Rightarrow \rho^2(A)=|\lambda_A|^2< \frac{q^q}{(q+1)^{(q+1)}}.
\end{split} 
\end{equation*}
Similarly we have 
\begin{equation*}
\rho^2(B)=|\lambda_B|^2< \frac{q^q}{(q+1)^{(q+1)}}.
\end{equation*}
\end{proof}
 
Let $k= {\lambda}_1(Q)>1$. Then the equation (\ref{eq:1}) can be written as 
\begin{equation*}
X^s+A^*X^{-t}A+B^*X^{-p}B=k\tilde{Q},
\end{equation*}
where $\tilde{Q}\leq I$ and $Q=k\tilde{Q}$, which implies that
\begin{equation}\label{equ:2.1}
k^{-1}X^s+k^{-\frac{1}{2}}A^*X^{-t}Ak^{-\frac{1}{2}}+k^{-\frac{1}{2}}B^*X^{-p}Bk^{-\frac{1}{2}}=\tilde{Q}.
\end{equation}
Let $k^{-\frac{1}{s}}X=\tilde{X}$, $\tilde{A}=k^{-\frac{1}{2}(\frac{t}{s}+1)}A$ and $\tilde{B}=k^{-\frac{1}{2}(\frac{p}{s}+1)}B$. Then the equation (\ref{equ:2.1}) reduces to
\begin{equation}\label{equ:3}
\tilde{X}^s+\tilde{A}^*\tilde{X}^{-t}\tilde{A}+\tilde{B}^*\tilde{X}^{-p}\tilde{B}=\tilde{Q},
\end{equation}	
where $s,t,p \geq 1$ $\tilde{A},\tilde{B} \in GL(n)$, $\tilde{Q} \in \mathcal{P}(n)$ with $\lambda_1 (\tilde{Q})\leq 1$.\\
Thus by using Theorem \ref{theorem:1} we conclude that if equation (\ref{equ:3}) has a Hermitian positive definite solution then $\rho^2({\tilde{A}})<\frac{q^q}{(q+1)^{(q+1)}}$ and $\rho^2({\tilde{B}})<\frac{q^q}{(q+1)^{(q+1)}}$, where $q=\min\{\frac{t}{s}, \frac{p}{s}\}$, implies that $\rho^2({A})<\frac{q^qk^{(1+\frac{t}{s})}}{(q+1)^{(q+1)}}\leq\frac{q^qk^{(1+\tilde{q})}}{(q+1)^{(q+1)}}$ and $\rho^2({B})<\frac{q^qk^{(1+\frac{p}{s})}}{(q+1)^{(q+1)}}\leq\frac{q^qk^{(1+\tilde{q})}}{(q+1)^{(q+1)}}$, where $\tilde{q}=\max\{\frac{t}{s},\frac{p}{s}\}$. Thus we have the following theorem.
\begin{theorem}\label{theorem:2}
	Let $k={\lambda}_1(Q)> 1$. If equation (\ref{eq:1}) has a Hermitian positive definite solution then $\rho^2(A)< \frac{q^qk^{(1+\tilde{q})}}{(q+1)^{(q+1)}}$ and $\rho^2(B)< \frac{q^qk^{(1+\tilde{q})}}{(q+1)^{(q+1)}}$, where $q=\min\{\frac{t}{s}, \frac{p}{s}\}$ and $\tilde{q}=\max\{\frac{t}{s}, \frac{p}{s}\}$.
\end{theorem}
Now we give some sufficient conditions for the existence of Hermitian positive definite solution of (\ref{eq:1}).
\begin{theorem}\label{theorem:3}
	Let $k={\lambda}_1(Q)\leq 1$. Then equation (\ref{eq:1}) has a Hermitian positive definite solution in $[(\frac{q\tilde{k}}{q+1})^\frac{1}{s}I,Q^\frac{1}{s}]$ if $||A||^2+||B||^2< \frac{q^{\tilde{q}}\tilde{k}^{(\tilde{q}+1)}}{(q+1)^{(\tilde{q}+1)}}$, where $\tilde{k}={\lambda}_n(Q)$, $q=\min\{\frac{t}{s}, \frac{p}{s}\}$ and $\tilde{q}=\max\{\frac{t}{s}, \frac{p}{s}\}$.
\end{theorem}
\begin{proof}
First notice that $||A||^2+||B||^2< \frac{q^{\tilde{q}}\tilde{k}^{(\tilde{q}+1)}}{(q+1)^{(\tilde{q}+1)}}$ implies
\begin{equation*}
\begin{split}
 \rho^2(A)+\rho^2(B)\leq ||A||^2+||B||^2&<\frac{q^{\tilde{q}}\tilde{k}^{(\tilde{q}+1)}}{(q+1)^{(\tilde{q}+1)}}\\
 &=\frac{q^q}{(q+1)^{(q+1)}}\tilde{k}^{(\tilde{q}+1)}\Big(\frac{q}{q+1}\Big)^{(\tilde{q}-q)}\\
 &<\frac{q^q}{(q+1)^{(q+1)}}, ~~~~~~(\textmd{ as  }\tilde{k}<1 \textmd{ and }\tilde{q}\geq q). 
\end{split}
\end{equation*}

Now to prove our claim, first we show that if $||A||^2+||B||^2< \frac{q^{\tilde{q}}\tilde{k}^{(\tilde{q}+1)}}{(q+1)^{(\tilde{q}+1)}}$ then the equation (\ref{equ:2}) has a solution in $\Big[\big(\frac{q\tilde{k}}{q+1}\big)I,Q\Big]$. Let $f(Y)=Q-A^*Y^{-\frac{t}{s}}A-B^*Y^{-\frac{p}{s}}B$. Then $f$ is continuous in $\Big[\big(\frac{q\tilde{k}}{q+1}\big)I,Q\Big]$ and for any  $Y\in\Big[\big(\frac{q\tilde{k}}{q+1}\big)I,Q\Big]$, we have $f(Y)\leq Q$ and 
\begin{equation*}
\begin{split}
f(Y)&=Q-A^*Y^{-\frac{t}{s}}A-B^*Y^{-\frac{p}{s}}B\\
&\geq \tilde{k}I-\Big(\frac{q\tilde{k}}{q+1}\Big)^{-\frac{t}{s}}A^*A-\Big(\frac{q\tilde{k}}{q+1}\Big)^{-\frac{p}{s}}B^*B\\
&\geq \tilde{k}I-\Big(\frac{q\tilde{k}}{q+1}\Big)^{-\tilde{q}}A^*A-\Big(\frac{q\tilde{k}}{q+1}\Big)^{-\tilde{q}}B^*B,~~ \Big(\textmd{as ~~}\frac{q\tilde{k}}{q+1}<1 \textmd{~~and ~~}\tilde{q}=\max\Big\{\frac{t}{s},\frac{p}{s}\Big\}\Big)\\
&=\tilde{k}I-\Big(\frac{q\tilde{k}}{q+1}\Big)^{-\tilde{q}}\Big[A^*A+B^*B\Big]\\
&>\tilde{k}I-\Big(\frac{q\tilde{k}}{q+1}\Big)^{-\tilde{q}}\frac{q^{\tilde{q}}\tilde{k}^{(\tilde{q}+1)}}{(q+1)^{(\tilde{q}+1)}}I\\
&=\Big(\frac{q\tilde{k}}{q+1}\Big)I.
\end{split}
\end{equation*}
Thus $f$ maps $\Big[\big(\frac{q\tilde{k}}{q+1}\big)I,Q\Big]$ into itself. Therefore by using Brouwer's fixed point theorem we conclude that $f$ has fixed point $\bar{Y}$ in  $\Big[\big(\frac{q\tilde{k}}{q+1}\big)I,Q\Big]$, which is in fact a solution of equation (\ref{equ:2}).\\
Now, if $\bar{Y}\in\Big[\big(\frac{q\tilde{k}}{q+1}\big)I,Q\Big]$ is a solution of (\ref{equ:2}), then  ${\bar{Y}}^\frac{1}{s}\in\Big[\big(\frac{q\tilde{k}}{q+1}\big)^\frac{1}{s}I,Q^\frac{1}{s}\Big]$ is a solution of (\ref{eq:1}). Therefore, if  $||A||^2+||B||^2< \frac{q^{\tilde{q}}\tilde{k}^{(\tilde{q}+1)}}{(q+1)^{(\tilde{q}+1)}}$ then the equation (\ref{eq:1}) has a solution in $\Big[\big(\frac{q\tilde{k}}{q+1}\big)^\frac{1}{s}I,Q^\frac{1}{s}\Big]$. 
\end{proof}

\begin{theorem}
		Let $k={\lambda}_1(Q)>1$. Then equation (\ref{eq:1}) has a Hermitian positive definite solution in $[\big(\frac{q\tilde{k}}{k(q+1)}\big)^\frac{1}{s}I,Q^\frac{1}{s}]$ if $||A||^2+||B||^2< \frac{q^{\tilde{q}}\tilde{k}^{(\tilde{q}+1)}}{k^{\tilde{q}}(q+1)^{(\tilde{q}+1)}}$, where $\tilde{k}={\lambda}_n(Q)$, $q=\min\{\frac{t}{s}, \frac{p}{s}\}$ and $\tilde{q}=\max\{\frac{t}{s}, \frac{p}{s}\}$.
\end{theorem}
\begin{proof}
	The proof is similar to the proof of Theorem \ref{theorem:3}. Also notice that $||A||^2+||B||^2< \frac{q^{\tilde{q}}{\tilde{k}}^{(\tilde{q}+1)}}{k^{\tilde{q}}(q+1)^{(\tilde{q}+1)}}$ implies
	\begin{equation*}
	\begin{split}
	\rho^2(A)+\rho^2(B)\leq ||A||^2+||B||^2&<\frac{q^{\tilde{q}}{\tilde{k}}^{(\tilde{q}+1)}}{k^{\tilde{q}}(q+1)^{(\tilde{q}+1)}}\\
	&=\frac{q^q}{(q+1)^{(q+1)}}k^{(1+\tilde{q})}\Big(\frac{\tilde{k}}{k}\Big)^{(\tilde{q}+1)}\frac{1}{k^{\tilde{q}}}\Big(\frac{q}{q+1}\Big)^{(\tilde{q}-q)}\\
	&<\frac{q^q}{(q+1)^{(q+1)}}k^{(1+\tilde{q})}, ~~~~~~(\textmd{ as }k> 1, k\geq\tilde{k} \textmd{ and }\tilde{q}\geq q).
	\end{split}
	\end{equation*}
	
	 Let $f(Y)=Q-A^*Y^{-\frac{t}{s}}A-B^*Y^{-\frac{p}{s}}B$. Then $f$ is continuous in $\Big[\big(\frac{q\tilde{k}}{k(q+1)}\big)I,Q\Big]$ and for any  $Y\in\Big[\big(\frac{q\tilde{k}}{k(q+1)}\big)I,Q\Big]$, we have $f(Y)\leq Q$ and 
	\begin{equation*}
	\begin{split}
	f(Y)&=Q-A^*Y^{-\frac{t}{s}}A-B^*Y^{-\frac{p}{s}}B\\
	&\geq \tilde{k}I-\Big(\frac{q\tilde{k}}{k(q+1)}\Big)^{-\frac{t}{s}}A^*A-\Big(\frac{q\tilde{k}}{k(q+1)}\Big)^{-\frac{p}{s}}B^*B\\
	&\geq \tilde{k}I-\Big(\frac{q\tilde{k}}{k(q+1)}\Big)^{-\tilde{q}}A^*A-\Big(\frac{q\tilde{k}}{k(q+1)}\Big)^{-\tilde{q}}B^*B,\\~ 
	& \Big(\textmd{as ~}\frac{q\tilde{k}}{k(q+1)}<1 \textmd{~and ~}\tilde{q}=\max\Big\{\frac{t}{s},\frac{p}{s}\Big\}\Big)\\
	&=\tilde{k}I-\Big(\frac{q\tilde{k}}{k(q+1)}\Big)^{-\tilde{q}}\Big[A^*A+B^*B\Big]\\
	&>\tilde{k}I-\Big(\frac{q\tilde{k}}{k(q+1)}\Big)^{-\tilde{q}}\frac{q^{\tilde{q}}\tilde{k}^{(\tilde{q}+1)}}{k^{\tilde{q}}(q+1)^{(\tilde{q}+1)}}I\\
	&=\Big(\frac{q\tilde{k}}{q+1}\Big)I>\Big(\frac{q\tilde{k}}{k(q+1)}\Big)I.
	\end{split}
	\end{equation*} 
	Thus $f$ maps $\Big[\big(\frac{q\tilde{k}}{k(q+1)}\big)I,Q\Big]$ into itself. Therefore by using Brouwer's fixed point theorem we conclude that $f$ has fixed point $\bar{Y}$ in  $\Big[\big(\frac{q\tilde{k}}{k(q+1)}\big)I,Q\Big]$, which is in fact a solution of equation (\ref{equ:2}).\\
	Now, if $\bar{Y}\in\Big[\big(\frac{q\tilde{k}}{k(q+1)}\big)I,Q\Big]$ is a solution of (\ref{equ:2}), then  ${\bar{Y}}^\frac{1}{s}\in\Big[\big(\frac{q\tilde{k}}{k(q+1)}\big)^\frac{1}{s}I,Q^\frac{1}{s}\Big]$ is a solution of (\ref{eq:1}). Therefore, if $||A||^2+||B||^2< \frac{q^{\tilde{q}}\tilde{k}^{(\tilde{q}+1)}}{k^{\tilde{q}}(q+1)^{(\tilde{q}+1)}}$, then the equation (\ref{eq:1}) has a solution in $\Big[\big(\frac{q\tilde{k}}{k(q+1)}\big)^\frac{1}{s}I,Q^\frac{1}{s}\Big]$. 
\end{proof}

Note that if (\ref{eq:1}) has a Hermitian positive definite solution $X$, then 
\begin{equation*}
\begin{split}
&~~~~A^*X^{-t}A < X^s+A^*X^{-t}A+B^*X^{-p}B=Q\\
&\Rightarrow X^{-t} < A^{-*}QA^{-1} \hspace{.3in}(\textmd{since }A \leq B \Rightarrow D^*AD \leq D^*BD)\\
&\Rightarrow X^t > AQ^{-1}A^* \hspace{.5in} (\mbox{By lemma \ref{lemma 4}})\\
&\Rightarrow X > (AQ^{-1}A^*)^\frac{1}{t} \geq \lambda_n^\frac{1}{t}(AQ^{-1}A^*)I.
\end{split}
\end{equation*}
Similarly we get $X > (BQ^{-1}B^*)^\frac{1}{p} \geq \lambda_n^\frac{1}{p}(BQ^{-1}B^*)I$.
Thus we have
\begin{equation*}
X \geq \max \{\lambda_n^\frac{1}{t}(AQ^{-1}A^*),~\lambda_n^\frac{1}{p}(BQ^{-1}B^*)\}I=cI~~\textmd{(say)}.
\end{equation*}
Also since $X^s < Q \Rightarrow X < Q^\frac{1}{s}$, therefore we get $X \in [cI,Q^\frac{1}{s}]$.

\begin{theorem} \label{thm 1}
	If equation (\ref{eq:1}) has a Hermitian positive definite solution $X$, then $X \in [mI,N]$, where $m= \max\{\lambda_n^\frac{1}{t}(A'),~\lambda_n^\frac{1}{p}(B')\}$,\\ $N=[Q-\Big(\frac{\lambda_n(Q^{-1})}{\lambda_1(Q^{-1})}\Big)^\frac{t-1}{s}A^*Q^{-\frac{t}{s}}A-\Big(\frac{\lambda_n(Q^{-1})}{\lambda_1(Q^{-1})}\Big)^\frac{p-1}{s}B^*Q^{-\frac{p}{s}}B]^\frac{1}{s}$,\\ $A'=A(Q-c^sI)^{-1}A^*$, $B'=B(Q-c^sI)^{-1}B^*$\\ and $c=\max \{\lambda_n^\frac{1}{t}(AQ^{-1}A^*),~\lambda_n^\frac{1}{p}(BQ^{-1}B^*)\}$.
\end{theorem}

\begin{proof}
Let $X$ be a Hermitian positive definite solution of equation (\ref{eq:1}), then $cI \leq X \leq Q^\frac{1}{s}$ implies
\begin{equation}\label{eq:5}
Q^{-\frac{1}{s}} \leq X^{-1} \leq c^{-1}I.
\end{equation}

\begin{equation}\label{eq:6}
\lambda_n^{\frac{1}{s}}(Q^{-1})I \leq Q^{-\frac{1}{s}} \leq \lambda_1^{\frac{1}{s}}(Q^{-1})I
\end{equation}
Thus using Lemma \ref{lemma 2}, equation (\ref{eq:5}) and (\ref{eq:6}) we get
\begin{equation}\label{eq:7}
Q^{-\frac{t}{s}} \leq \Big(\frac{\lambda_1(Q^{-1})}{\lambda_n(Q^{-1})}\Big)^\frac{t-1}{s} X^{-t}.
\end{equation}
Similarly we also get
\begin{equation}\label{eq:8}
Q^{-\frac{p}{s}} \leq \Big(\frac{\lambda_1(Q^{-1})}{\lambda_n(Q^{-1})}\Big)^\frac{p-1}{s} X^{-p}.
\end{equation}
Therefore from (\ref{eq:1}) we have 
\begin{equation*}
\begin{split}
Q-X^s&= A^*X^{-t}A+B^*X^{-p}B\\
&\geq \Big(\frac{\lambda_n(Q^{-1})}{\lambda_1(Q^{-1})}\Big)^\frac{t-1}{s} A^*Q^{-\frac{t}{s}}A+\Big(\frac{\lambda_n(Q^{-1})}{\lambda_1(Q^{-1})}\Big)^\frac{p-1}{s} B^*Q^{-\frac{p}{s}}B \notag  ~~(\mbox{from~ (\ref{eq:7}), (\ref{eq:8})}) \\
&\hspace*{-1.6cm}\Rightarrow X \leq [Q-\Big(\frac{\lambda_n(Q^{-1})}{\lambda_1(Q^{-1})}\Big)^\frac{t-1}{s} A^*Q^{-\frac{t}{s}}A-\Big(\frac{\lambda_n(Q^{-1})}{\lambda_1(Q^{-1})}\Big)^\frac{p-1}{s} B^*Q^{-\frac{p}{s}}B]^\frac{1}{s}=N \textmd{(say)}.
\end{split}
\end{equation*}
Now,
\begin{equation*}
\begin{split}
Q-A^*X^{-t}A \geq X^s&\Rightarrow A^*X^{-t}A \leq Q-X^s\\
&\Rightarrow A^{-1}X^tA^{-*} \geq (Q-X^s)^{-1}\\
&\Rightarrow X \geq (A(Q-X^s)^{-1}A^*)^\frac{1}{t}.
\end{split}
\end{equation*}
Similarly we have $X \geq (B(Q-X^s)^{-1}B^*)^\frac{1}{p}$.\\
Again,
\begin{equation*}
\begin{split}
X \geq cI\Rightarrow X^s \geq c^sI&\Rightarrow Q-X^s \leq Q-c^sI\\
&\Rightarrow (A(Q-X^s)^{-1}A^*)^\frac{1}{t} \geq (A(Q-c^sI)^{-1}A^*)^\frac{1}{t}.
\end{split}
\end{equation*}
Therefore we have $X \geq (A(Q-c^sI)^{-1}A^*)^\frac{1}{t} \geq \lambda_n^\frac{1}{t}(A')I$,\\ where $A'=A(Q-c^sI)^{-1}A^*$.\\
Similarly we also have $X \geq (B(Q-c^sI)^{-1}B^*)^\frac{1}{p} \geq \lambda_n^\frac{1}{p}(B')I$,\\  where $B'=B(Q-c^sI)^{-1}B^*$.\\
Thus, $X \geq \max\{\lambda_n^\frac{1}{t}(A'),~\lambda_n^\frac{1}{p}(B')\}I=mI$. So $X \in [mI,~N]$.
\end{proof}

\begin{remark}
	$[mI,N] \subseteq [cI,Q^\frac{1}{s}]$.
\end{remark}

\begin{proof}
\begin{equation*}
\begin{split}
&Q-c^sI \leq Q\\
&\Rightarrow (A(Q-c^sI)^{-1}A^*)^\frac{1}{t} \geq (AQ^{-1}A^*)^\frac{1}{t} \geq \lambda_n^\frac{1}{t}(AQ^{-1}A^*)I\\
&\Rightarrow \max\{\lambda_n^\frac{1}{t}(A'), \lambda_n^\frac{1}{p}(B')\}I \geq \lambda_n^\frac{1}{t}(A')I \geq \lambda_n^\frac{1}{t}(AQ^{-1}A^*)I.
\end{split}
\end{equation*}
Also, $\max\{\lambda_n^\frac{1}{t}(A'), \lambda_n^\frac{1}{p}(B')\}I \geq \lambda_n^\frac{1}{p}(B')I \geq \lambda_n^\frac{1}{p}(BQ^{-1}B^*)I$.\\
Therefore, $mI=\max\{\lambda_n^\frac{1}{t}(A'),~\lambda_n^\frac{1}{p}(B')\}I \geq \max\{\lambda_n^\frac{1}{t}(AQ^{-1}A^*),~\lambda_n^\frac{1}{p}(BQ^{-1}B^*)\}I \\ =cI $.\\
Also since, $N=[Q-(\frac{\lambda_n(Q^{-1})}{\lambda_1(Q^{-1})})^\frac{t-1}{s} A^*Q^{-\frac{t}{s}}A-(\frac{\lambda_n(Q^{-1})}{\lambda_1(Q^{-1})})^\frac{p-1}{s} B^*Q^{-\frac{p}{s}}B]^\frac{1}{s} \leq Q^\frac{1}{s}$, 
 we have $[mI,N] \subseteq [cI,Q^\frac{1}{s}]$.
\end{proof}
  
  In next couple of theorems we give some sufficient criteria for the uniqueness of solutions of (\ref{eq:1}).
\begin{theorem}\label{thm 2}
	If $(AQ^{-1}A^*)^\frac{s}{t}+(BQ^{-1}B^*)^{\frac{s}{p}} \leq Q$, $A^*X^{-t}A+B^*X^{-p}B \leq Q-(AQ^{-1}A^*)^\frac{s}{t}-(BQ^{-1}B^*)^\frac{s}{p}$, for all $X \in [cI,Q^\frac{1}{s}]$ \\and 
	$\frac{1}{s}a^{\frac{1}{s}-1}[\frac{t}{c^{t+1}}\|A\|^2+\frac{p}{c^{p+1}}\|B\|^2] < 1$, where $a=\lambda_n^{\frac{s}{t}}(AQ^{-1}A^*)+\lambda_n^{\frac{s}{p}}(BQ^{-1}B^*)$ and $c=\max \{\lambda_n^\frac{1}{t}(AQ^{-1}A^*),~\lambda_n^\frac{1}{p}(BQ^{-1}B^*)\}$.  Then equation (\ref{eq:1}) has a unique Hermitian positive definite solution in $[cI,Q^\frac{1}{s}]$, and hence in $\mathcal{P}(n)$.
\end{theorem} 

\begin{proof}
	First of all note that $[(AQ^{-1}A^*)^\frac{s}{t}+(BQ^{-1}B^*)^\frac{s}{p}]^\frac{1}{s} \geq ((AQ^{-1}A^*)^\frac{s}{t})^\frac{1}{s}=(AQ^{-1}A^*)^\frac{1}{t} \geq \lambda_n^\frac{1}{t}(AQ^{-1}A^*)I$,\\ and $[(AQ^{-1}A^*)^\frac{s}{t}+(BQ^{-1}B^*)^\frac{s}{p}]^\frac{1}{s} \geq (BQ^{-1}B^*)^\frac{1}{p} \geq \lambda_n^\frac{1}{p}(BQ^{-1}B^*)I$.\\
	 Thus we have 
	  \begin{equation}\label{eq:9}
	 [(AQ^{-1}A^*)^\frac{s}{t}+(BQ^{-1}B^*)^\frac{s}{p}]^\frac{1}{s} \geq \max \{\lambda_n^\frac{1}{t}(AQ^{-1}A^*),~\lambda_n^\frac{1}{p}(BQ^{-1}B^*)\}I=cI.
	 \end{equation}
	 Now let $X \in [cI,Q^\frac{1}{s}]$ and define a function $f:[cI,Q^\frac{1}{s}] \to \mathcal{M}(n)$ (the set of all $n\times n$ matrices) by  $f(X)=(Q-A^*X^{-t}A-B^*X^{-p}B)^\frac{1}{s}$. Therefore 
	 \begin{equation*}
	 \begin{split}
	 Q^\frac{1}{s} &\geq (Q-A^*X^{-t}A-B^*X^{-p}B)^\frac{1}{s}=f(X)\\
	 & \geq  (Q-Q+(AQ^{-1}A^*)^\frac{s}{t}+(BQ^{-1}B^*)^\frac{s}{p})^\frac{1}{s}  \hspace{.3in} \textmd{(by our assumption)}\\
	 &= [(AQ^{-1}A^*)^\frac{s}{t}+(BQ^{-1}B^*)^\frac{s}{p}]^\frac{1}{s}\\
	 & \geq  cI \notag \hspace{2.6in} (\mbox{from~(\ref{eq:9})}).
	 \end{split}
	 \end{equation*}
Thus $f$ maps $[cI,Q^\frac{1}{s}]$ into itself. Now let $X,Y \in [cI,Q^\frac{1}{s}]$, then
\begin{equation*}
\begin{split}
Q-(A^*X^{-t}A+B^*X^{-p}B) & \geq (AQ^{-1}A^*)^\frac{s}{t}+(BQ^{-1}B^*)^{\frac{s}{p}}\\
& \geq \Big(\lambda_n^{\frac{s}{t}}(AQ^{-1}A^*)+\lambda_n^{\frac{s}{p}}(BQ^{-1}B^*)\Big)I=aI.
\end{split}
\end{equation*}
Similarly, $Q-(A^*Y^{-t}A+B^*Y^{-p}B) \geq aI$. Thus for $X,Y \in [cI,Q^\frac{1}{s}]$, we have
\begin{equation*}
\begin{split}
\|f(X)-f(Y)\| &= \|(Q-A^*X^{-t}A-B^*X^{-p}B)^\frac{1}{s}-(Q-A^*Y^{-t}A-B^*Y^{-p}B)^\frac{1}{s}\| \\
& \leq \frac{1}{s}a^{\frac{1}{s}-1}\|A^*(Y^{-t}-X^{-t})A+B^*(Y^{-p}-X^{-p})B\| \hspace{.1in} (\textmd{By ~Lemma~\ref{lemma 3}})\\ 
& \leq  \frac{1}{s}a^{\frac{1}{s}-1} \Big[\|A\|^2 \|Y^{-t}-X^{-t}\|+\|B\|^2 \|Y^{-p}-X^{-p}\|\Big] \\
&\leq \frac{1}{s}a^{\frac{1}{s}-1} \Big[\|A\|^2\|\sum_{i=1}^t Y^{-(t+1)+i}(X-Y)X^{-i}\|\\
&\hspace*{0.5cm}+\|B\|^2\|\sum_{j=1}^p Y^{-(p+1)+j}(X-Y)X^{-j}\|\Big]\\
&\leq \frac{1}{s}a^{\frac{1}{s}-1} \Big[\|A\|^2\sum_{i=1}^t \big(\|Y^{-(t+1)+i}\|~\|X^{-i}\|~\|(X-Y)\|\big)\\
&\hspace*{0.5cm}+\|B\|^2\sum_{j=1}^p \big(\|Y^{-(p+1)+j}\|~\|X^{-j}\|~\|(X-Y)\|\big)\Big]\\
&\leq \frac{1}{s}a^{\frac{1}{s}-1} \Big[\|A\|^2\sum_{i=1}^t (c^{-(t+1)+i}~c^{-i})\\
&\hspace*{0.5cm}+\|B\|^2 \sum_{j=1}^p (c^{-(p+1)+j}~c^{-j})\Big]\|X-Y\|\\
&\leq\frac{1}{s} a^{\frac{1}{s}-1} [~\|A\|^2 ~\frac{t}{c^{t+1}}+\|B\|^2~\frac{p}{c^{p+1}}~]~\|X-Y\|.
\end{split}
\end{equation*}
Therefore by Banach's contraction principle $f$ has a unique fixed point in $[cI,Q^\frac{1}{s}]$, which is a Hermitian positive definite solution of equation (\ref{eq:1}). Also since any solutions of (\ref{eq:1}) must be in $[cI,Q^\frac{1}{s}]$, the solution is unique in $\mathcal{P}(n)$.
\end{proof}

\begin{theorem}\label{thm 3}
	If there exists $k > 0$ such that $\frac{1}{k^t}+\frac{1}{k^p}<1$ and 
	$\lambda_1 (c_1^sQ^{-1}) \leq (1-k^{-t}-k^{-p})k^{-s}$, where $c_1=\max \{\lambda_1^\frac{1}{t}(AQ^{-1}A^*),~\lambda_1^\frac{1}{p}(BQ^{-1}B^*)\}$, then $kc_1I \leq Q^{\frac{1}{s}}$ and equation (\ref{eq:1}) has a Hermitian positive definite solution in $[kc_1I,Q^\frac{1}{s}]$. Furthermore, the solution is unique if  $\frac{1}{s}(kc_1)^{1-s}[\frac{t}{(kc_1)^{t+1}} \|A\|^2+\frac{p}{(kc_1)^{p+1}} \|B\|^2] < 1$.
\end{theorem}

\begin{proof}
 Since $c_1^sQ^{-1} \leq \lambda_1 (c_1^sQ^{-1})I \leq (1-k^{-t}-k^{-p})k^{-s}I$, we have 
$(kc_1)^sI \leq (1-k^{-t}-k^{-p})Q \leq Q$, which implies
\begin{equation}\label{eq:11}
kc_1I \leq Q^\frac{1}{s}.
\end{equation}
This proves our first claim.\\
Now let $X \in [kc_1I,Q^\frac{1}{s}]$. Then
\begin{equation*}
\begin{split}
X^t &\geq k^tc_1^tI\\
&= k^t \max \{\lambda_1(AQ^{-1}A^*), \lambda_1^\frac{t}{p}(BQ^{-1}B^*)\}I\\
&\geq k^t \lambda_1(AQ^{-1}A^*)I\geq k^t AQ^{-1}A^*\\
&\hspace*{-1cm} \Rightarrow X^{-t} \leq k^{-t} A^{-*}QA^{-1}.
\end{split}
\end{equation*}
Similarly, $X^{-p} \leq k^{-p} B^{-*}QB^{-1}$.\\
Define a function $G:[kc_1I,Q^\frac{1}{s}]\to \mathcal{M}(n)$ by  $G(X)=\big(Q-A^*X^{-t}A-B^*X^{-p}B\big)^\frac{1}{s}$.
Therefore
\begin{equation*}
\begin{split}
G(X) &\geq \big(Q-A^* (k^{-t} A^{-*}QA^{-1} )A-B^*(k^{-p}B^{-*}QB^{-1})B\big)^\frac{1}{s} \\
&=\big((1-k^{-t}-k^{-p})Q\big)^\frac{1}{s} \\
&\geq kc_1I \hspace{.5in} (\mbox{from~ (\ref{eq:11})}).
\end{split}
\end{equation*}
Also $G(X)=(Q-A^*X^{-t}A-B^*X^{-p}B)^\frac{1}{s} \leq Q^{\frac{1}{s}}$.\\
Therefore $G$ maps $[kc_1I,Q^\frac{1}{s}]$ into itself and $G$ is continuous. Also $[kc_1I,Q^\frac{1}{s}]$ is closed and convex. Thus using Brouwer's fixed point theorem we conclude that $G$ has a fixed point in $[kc_1I,Q^\frac{1}{s}]$, which is a Hermitian positive definite solution of (\ref{eq:1}). 

Now for all $X \in [kc_1I,Q^\frac{1}{s}]$, we have $(Q-A^*X^{-t}A-B^*X^{-p}B)^\frac{1}{s} \geq kc_1I \Rightarrow Q-A^*X^{-t}A-B^*X^{-p}B \geq (kc_1)^sI$.\\
Thus progressing as in latter half of Theorem \ref{thm 2}, we can conclude by Banach's contraction principle that the solution is unique if  $\frac{1}{s}(kc_1)^{1-s}[\frac{t}{(kc_1)^{t+1}} \|A\|^2+\frac{p}{(kc_1)^{p+1}} \|B\|^2] < 1$.
\end{proof}

Now with a different perspective, we give a necessary and sufficient condition for the existence of a Hermitian positive definite solution of (\ref{eq:1}).
\begin{theorem}\label{thm 4}
	Equation (\ref{eq:1}) has a Hermitian positive definite solution if and only if $A$ and $B$ can be factored as $A=(U \Lambda U^*)^{\frac{t}{2s}} N_1$ and $B=(U \Lambda U^*)^{\frac{p}{2s}}N_2$, where $U$ is a unitary and $\Lambda$ is a diagonal matrix and
	$ \left(
	\begin{array}{c}
	\Lambda^{\frac{1}{2}} U^* Q^{-\frac{1}{2}} \\
	N_1 Q^{-\frac{1}{2}} \\
	N_2 Q^{-\frac{1}{2}}
	\end{array} \right)$  is column orthonormal.
\end{theorem}
\begin{proof}
As the proof is similar to the proof of Theorem 2.7 of Liu-Chen  \cite{LC2011}, we exclude the proof. 
\end{proof}
Next we derive some iterations (with examples) to compute the Hermitian positive definite solutions of  (\ref{eq:1}).\\
Let $s \geq t$ and without loss of generality let $t \geq p$. Then $\max \{s,t,p\}=s$. Therefore taking $X^s=Y$ in (\ref{eq:1}) we get
\begin{equation}\label{eq:15}
Y+A^*Y^{-\frac{t}{s}}A+B^*Y^{-\frac{p}{s}}B=Q
\end{equation}
where $\frac{t}{s},\frac{p}{s} \leq 1$. Thus finding a Hermitian positive definite  solution of equation (\ref{eq:15}) (say $\bar{Y}$) will lead us to a Hermitian positive definite solution of equation (\ref{eq:1}) (which will be $\bar{Y}^\frac{1}{s}$). Also note that if $Y$ is any Hermitian positive definite solution of equation (\ref{eq:15}) then $Y\in[c^sI,Q]$.
\begin{theorem}\label{thm 6}
	Assume that there exists $\alpha \in (0,\lambda_n(Q)]$, such that
	\begin{equation}\label{eq:16}
	\alpha+\alpha^{-\frac{t}{s}}\|A\|^2+\alpha^{-\frac{p}{s}}\|B\|^2 < \lambda_n(Q).
	\end{equation}
	Let $\beta=\lambda_n(Q-\alpha^{-\frac{t}{s}}A^*A-\alpha^{-\frac{p}{s}}B^*B)$. Then equation (\ref{eq:15}) has a unique Hermitian positive definite solution $\bar{Y} \in [\alpha I,Q]$ if 
	\begin{center}
		$t\beta^{-\frac{t}{s}}\|A\|^2+p\beta^{-\frac{p}{s}}\|B\|^2 < s\beta$.
	\end{center}
 In this case, the sequence $Y_n$ defined by 
 \begin{center}
 	$Y_0=\alpha I$,\\
 	$Y_{n+1}=Q-A^*Y_n^{-\frac{t}{s}}A-B^* Y_n^{-\frac{p}{s}}B $;\
 	\end{center}
 	converge to $\bar{Y}$  and the error estimation is given by
 	\begin{equation*}
 	\|Y_n-\bar{Y}\|\leq \frac{\delta^n}{1-\delta}\max \{\|Y_1-Y_0\|\},
 	\end{equation*}
 	where $\delta=\frac{t}{s} \|A\|^2 \beta^{-\frac{t}{s}-1}+\frac{p}{s}\|B\|^2\beta^{-\frac{p}{s}-1}$.
\end{theorem}
\begin{proof}
  Define a function $f:[ \alpha I,Q]\to\mathcal{M}(n)$ by $f(Y)=Q-A^*Y^{-\frac{t}{s}}A-B^*Y^{-\frac{p}{s}}B$. Then $f(\alpha I)\geq (\lambda_n(Q)
-\alpha^{-\frac{t}{s}}\|A\|^2-\alpha^{-\frac{p}{s}}\|B\|^2)I > \alpha I$. Therefore $0<\alpha I<f(\alpha I)$, implies $\beta \geq \alpha > 0$. Also for any  $Y \in [\alpha I,Q]$, $ A^*Y^{-\frac{t}{s}}A \leq \alpha^{-\frac{t}{s}}A^*A$ and $B^*Y^{-\frac{p}{s}}A \leq \alpha^{-\frac{p}{s}}B^*B$.
Therefore $0 < \beta I \leq Q-\alpha^{-\frac{t}{s}}A^*A-\alpha^{-\frac{p}{s}}B^*B \leq Q-A^*Y^{-\frac{t}{s}}A-B^*Y^{-\frac{p}{s}}B \leq Q$. So $f$ maps $[\alpha I,Q]$ into $[\beta I,Q]$. More particularly, it maps $[\beta I,Q]$ into itself. Also it is easy to see that $f$ is order-preserving.
Now let $Y,Z \in [\beta I,Q]$ be such that $Y \leq Z$. Then by Lemma \ref{lemma 3} we have
\begin{center}
$~\|Y^{\frac{t}{s}}-Z^{\frac{t}{s}}\| \leq \frac{t}{s} \beta^{\frac{t}{s}-1}\|Y-Z\|,$\\
$\|Z^{-\frac{t}{s}}-Y^{-\frac{t}{s}}\| \leq \frac{t}{s} \beta^{-\frac{t}{s}-1}\|Y-Z\|,$\\
$\|Z^{-\frac{p}{s}}-Y^{-\frac{p}{s}}\| \leq \frac{p}{s} \beta^{-\frac{p}{s}-1}\|Y-Z\|.$
\end{center}
Therefore
\begin{equation*}
\begin{split}
\|f(Y)-f(Z)\|&=\|Q-A^*Y^{-\frac{t}{s}}A-B^*Y^{-\frac{p}{s}}B-Q+A^*Z^{-\frac{t}{s}}A+B^*Z^{-\frac{p}{s}}B\| \\
&=\|A^*(Z^{-\frac{t}{s}}-Y^{-\frac{t}{s}})A+B^*(Z^{-\frac{p}{s}}-Y^{-\frac{p}{s}})B\|\\
&\leq \|A\|^2\|Z^{-\frac{t}{s}}-Y^{-\frac{t}{s}}\|+\|B\|^2\|Z^{-\frac{p}{s}}-Y^{-\frac{p}{s}}\|\\
&\leq  (\frac{t}{s} \|A\|^2 \beta^{-\frac{t}{s}-1}+\frac{p}{s}\|B\|^2\beta^{-\frac{p}{s}-1})\|Y-Z\|,
\end{split}
\end{equation*} 
where $\frac{t}{s} \|A\|^2 \beta^{-\frac{t}{s}-1}+\frac{p}{s}\|B\|^2\beta^{-\frac{p}{s}-1} < 1$. Also, since for any $X\in [\alpha I,Q]$ has a lower bound $\alpha I$, an upper bound $Q$ and $f(Q)\leq Q$, therefore by Theorem \ref{lemma 7} equation (\ref{eq:15}) has a unique Hermitian positive definite solution $\bar{Y}$ in $[\beta I,Q]$ and hence in $[\alpha I,Q]$. Thus equation (\ref{eq:1}) has a unique Hermitian positive definite solution $\bar{X}=\bar{Y}^\frac{1}{s}$ in $\big[\alpha^\frac{1}{s}I,Q^\frac{1}{s}\big]$.
\end{proof}	

\begin{remark}
	The obtained solution $\bar{Y}$ in Theorem \ref{thm 6} is the maximal Hermitian positive definite solution of (\ref{eq:15}).  
\end{remark}
\begin{proof}
Let $\bar{Y_1}$ be any other Hermitian positive definite solution of (\ref{eq:15}) such that $\bar{Y_1}\geq \bar{Y}$. Then $\bar{Y_1}\leq Q$. Subsequently, $\bar{Y_1}\in [\bar{Y},Q]\subset[\alpha I,Q]$. But according to Theorem \ref{thm 6}, $\bar{Y}$ is the only Hermitian positive definite solution in $[\alpha I,Q]$. Therefore $\bar{Y_1}=\bar{Y}$. Thus $\bar{Y}$ is the maximal Hermitian positive definite solution of (\ref{eq:15}). By similar explanation we also conclude that $\bar{X}=\bar{Y}^\frac{1}{s}$ is the maximal Hermitian positive definite solution of (\ref{eq:1}). 
\end{proof}
\begin{example}
	Consider the following nonlinear matrix equation 
	\begin{equation}\label{exeq:1}
	X^3+A^*X^{-2}A+B^*X^{-1}B=Q,
	\end{equation}
	where $Q=\left( \begin{array}{ccc}
	2 & 0 & 0 \\
	0 & 2 & 0 \\
	0 & 0 & 2 \end{array} \right)$, $A= \left( \begin{array}{ccc} 0.02 & -0.1 & -0.02 \\
	0.08 & -0.1 & 0.02 \\
	-0.06 & -0.12 & 0.14 
	\end{array} \right)$ \\ and $B=\left( \begin{array}{ccc} -0.04 & 0.01 & -0.02 \\
	0.05 & 0.07 & -0.013 \\
	0.011 & 0.09 & 0.06 
	\end{array} \right)$. \\
	Then $s=3\geq t=2 \geq p=1$. Therefore in this case (\ref{exeq:1}) can be written as
	\begin{equation}\label{exeq:2}
	Y+A^*Y^{-\frac{2}{3}}A+B^*Y^{-\frac{1}{3}}B=Q,
	\end{equation} 
	where $Y=X^3\Rightarrow X=Y^\frac{1}{3}$.\\
	 Let $\alpha=1$, then $\alpha+\alpha^{-\frac{t}{s}}\|A\|^2+\alpha^{-\frac{p}{s}}\|B\|^2=1.06191157562005 <2= \lambda_n(Q)$ and $\beta=\lambda_n(Q-\alpha^{-\frac{t}{s}}A^*A-\alpha^{-\frac{p}{s}}B^*B)=1.946624597494775$. Also $t\beta^{-\frac{t}{s}}\|A\|^2+p\beta^{-\frac{p}{s}}\|B\|^2=0.0716001214949702 < 5.839873792484324=s\beta$. Therefore all the hypothesis of Theorem \ref{thm 6}  are satisfied with $\alpha=1$. After $8$ iterations, we get a solution of equation (\ref{exeq:2}) in $[\alpha I,Q]$ as \\
	
	$\left( \begin{array}{ccc} 
	1.990011507887876 & -0.001460413784344 & 0.003932548667216 \\
	-0.001460413784344 & 1.967817699120152 & 0.007205717778742 \\
	0.003932548667216 & 0.007205717778742 & 1.983761175394282
	\end{array} \right)$.  \\
	
	with error $e=3.124677099290812 \times 10^{-16}$. Therefore equation (\ref{exeq:1}) has a Hermitian positive definite solution in $[\alpha^{\frac{1}{3}}I, Q^{\frac{1}{3}}]$ as\\
	
	 $\left( \begin{array}{ccc} 
	 1.257819473237711 & -0.000309853059784 & 0.000829790450201 \\
	 -0.000309853059784 & 1.253124679713870 & 0.001525655417243 \\
	 0.000829790450201 & 0.001525655417243 & 1.256499440822798
	 \end{array} \right)$.  \\

	 	\begin{figure}
		\centering
	 		\begin{minipage}{0.7\textwidth}
	 			\includegraphics[width=0.8\textwidth]{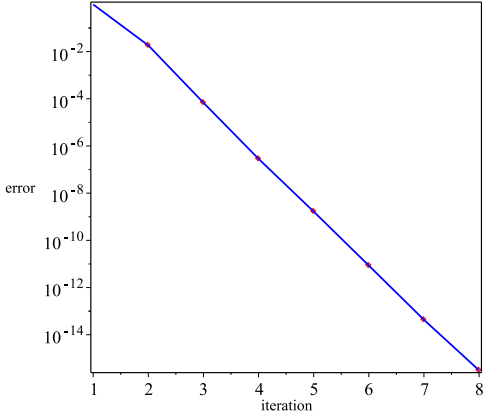}\\
	 			\caption{Convergence history of (\ref{exeq:2})}
	 			\label{Fig.1}
	 		\end{minipage}
	 	\end{figure}
	 \end{example}
Now let $t \geq s$ and without loss of generality let $t \geq p$, then $\max\{s,t,p\}=t$. Letting $X^t=Y$ equation (\ref{eq:1}) becomes
\begin{equation}\label{eq:17}
Y^{\frac{s}{t}}+A^*Y^{-1}A+B^*Y^{-\frac{p}{t}}B=Q,
\end{equation} 
where $\frac{s}{t},\frac{p}{t} \leq 1$. Again, a solution of equation (\ref{eq:17}) will yield a solution of our original equation. Now if equation (\ref{eq:17}) has a Hermitian positive definite solution $Y$ then
\begin{equation*}
\begin{split}
&~~~~A^*Y^{-1}A=Q-Y^{\frac{s}{t}}-B^*Y^{-\frac{p}{t}}B\\
&\Rightarrow Y^{-1}=A^{-*}(Q-Y^{\frac{s}{t}}-B^*Y^{-\frac{p}{t}}B)A^{-1}\\
&\Rightarrow Y=A(Q-Y^{\frac{s}{t}}-B^*Y^{-\frac{p}{t}}B)^{-1}A^*.
\end{split}
\end{equation*}
Since, $Y^{\frac{s}{t}}=Q-A^*Y^{-1}A-B^*Y^{-\frac{p}{t}}B \leq Q$ and $\lambda_n(Q)I \leq Q \leq \lambda_1(Q)I$, then by Lemma \ref{lemma 2} 
\begin{equation*}
Y \leq \Big(\frac{\lambda_1(Q)}{\lambda_n(Q)}\Big)^{\frac{t}{s}-1}Q^{\frac{t}{s}}=N_1~(\textmd{say})
\end{equation*}
Again, $A^*Y^{-1}A \leq Q\Rightarrow Y \geq AQ^{-1}A^* \geq \lambda_n(AQ^{-1}A^*)I$. Also since  $B^*Y^{-\frac{p}{t}}B \leq Q\Rightarrow Y^{\frac{p}{t}} \geq BQ^{-1}B^*$ and
$\lambda_1 (BQ^{-1}B^*)I \geq BQ^{-1}B^* \geq \lambda_n (BQ^{-1}B^*)I$, 

then again by the Lemma \ref{lemma 2},  
\begin{equation*}
\begin{split}
&~~~~(BQ^{-1}B^*)^{\frac{t}{p}} \leq \Big(\frac{\lambda_1 (BQ^{-1}B^*)}{\lambda_n (BQ^{-1}B^*)}\Big)^{\frac{t}{p}-1}Y\\
&\Rightarrow Y \geq \Big(\frac{\lambda_n (BQ^{-1}B^*)}{\lambda_1 (BQ^{-1}B^*)}\Big)^{\frac{t}{p}-1}(BQ^{-1}B^*)^{\frac{t}{p}} \geq \Big(\frac{\lambda_n (BQ^{-1}B^*)}{\lambda_1 (BQ^{-1}B^*)}\Big)^{\frac{t}{p}-1} \lambda_n^\frac{t}{p}(BQ^{-1}B^*)I.
\end{split}
\end{equation*}
Combining we get 
\begin{equation*}
Y \geq  \max \{\lambda_n(AQ^{-1}A^*),\Big(\frac{\lambda_n (BQ^{-1}B^*)}{\lambda_1 (BQ^{-1}B^*)}\Big)^{\frac{t}{p}-1} \lambda_n^\frac{t}{p}(BQ^{-1}B^*)\}I=m_1I~ \textmd{(say)}.
\end{equation*}
Thus $Y \in [m_1I,N_1]$.\\
\begin{theorem}\label{thm 7}
	Suppose there exists $b>0(\in \mathbb{R})$ such that the followings were considered: \\
	\begin{center}
		$(i) ~b> a$ \\
		$(ii) ~ Q \geq b^{-1}A^*A+b^{\frac{s}{t}}I+a^{-\frac{p}{t}}B^*B $ \\
		$(iii) ~s\|A\|^2 < \frac{1}{2}t \theta^2 a^{1-\frac{s}{t}}$ \\
		$(iv) ~p\|B\|^2 < sa^{\frac{p+s}{t}}$
	\end{center}
	where $a=\lambda_n(AQ^{-1}A^*)$ and $\theta=b^{-1}\lambda_n(A^*A)$. Then,
	\begin{enumerate}
		\item[$(I)$] equation (\ref{eq:17}) has a unique Hermitian positive definite solution $ \bar{X}\in [aI, bI]$,
		\item[$(II)$] $\bar{X} \in [F(aI,bI),F(bI,aI)]$,
		\item[$(III)$]  the sequences $X_n$ and $Y_n$ defined by 
		\begin{center}
			$X_0=aI$,\\
			$X_{n+1}=A(Q-X_n^\frac{s}{t}-B^* Y_n^{-\frac{p}{t}}B)^{-1}A^* $;\\
			$Y_0=bI$,\\
			$Y_{n+1}=A(Q-Y_n^\frac{s}{t}-B^* X_n^{-\frac{p}{t}}B)^{-1}A^*$,
		\end{center}
	\end{enumerate}
	converge to $\bar{X}$, that is $\displaystyle\lim_{n \rightarrow \infty} \|X_n-\bar{X}\|=\lim_{n \rightarrow \infty} \|Y_n-\bar{X}\|=0$ and the error estimation is given by
	\begin{equation*}
	\max \{\|X_n-\bar{X}\|,\|Y_n-\bar{X}\|\} \leq \frac{\delta^n}{1-\delta}\max \{\|X_1-X_0\|,\|Y_1-Y_0\|\},
	\end{equation*}
	where $\delta=2 \max\{\frac{s}{t}\|A\|^2~\theta^{-2}a^{\frac{s}{t}-1},~\frac{p}{t}\|A\|^2\|B\|^2\theta^{-2}a^{-\frac{p}{t}-1}\}.$
\end{theorem}
\begin{proof}
Let $X, Y \in [aI,bI]$,  then $Y \geq aI\Rightarrow Y^{-\frac{p}{t}} \leq a^{-\frac{p}{t}}I\Rightarrow B^*Y^{-\frac{p}{t}}B \leq a^{-\frac{p}{t}}B^*B$, as  $\frac{s}{t}$ and $\frac{p}{t} \leq 1$. Also $X \leq bI\Rightarrow X^{\frac{s}{t}} \leq b^{\frac{s}{t}}I$.
Therefore,
\begin{eqnarray}\label{eq:18}
Q-X^{\frac{s}{t}}-B^*Y^{-\frac{p}{t}}B &\geq & Q-b^{\frac{s}{t}}I-a^{-\frac{p}{t}}B^*B \notag \\
& \geq & b^{-1}A^*A > 0,~~~ \textmd{(By condition (ii))}
\end{eqnarray}    
Thus  $Q-X^{\frac{s}{t}}-B^*Y^{-\frac{p}{t}}B$ is invertible. Consider a function $F:[aI,bI] \times [aI,bI] \to \mathcal{M}(n)$ by \\ $F(X,Y)= A(Q-X^{\frac{s}{t}}-B^*Y^{-\frac{p}{t}}B)^{-1}A^*$. We will show that $F$ maps $[aI,bI] \times [aI,bI]$ into $[aI,bI]$. \\
Let $X,~Y \in [aI,bI]$. Then 
\begin{equation*}
\begin{split}
&~~~~Q-X^{\frac{s}{t}}-B^*Y^{-\frac{p}{t}}B \leq Q\\
&\Rightarrow A(Q-X^{\frac{s}{t}}-B^*Y^{-\frac{p}{t}}B)^{-1}A^* \geq AQ^{-1}A^*\\
&\Rightarrow F(X,Y)\geq AQ^{-1}A^* \geq \lambda_n(AQ^{-1}A^*)I=aI.
\end{split}
\end{equation*}
Also from (\ref{eq:18}) we get,
\begin{equation*}
\begin{split}
&~~~~Q-X^{\frac{s}{t}}-B^*Y^{-\frac{p}{t}}B \geq Q-b^{\frac{s}{t}}I-a^{-\frac{p}{t}}B^*B \geq b^{-1}A^*A\\
&\Rightarrow (Q-X^{\frac{s}{t}}-B^*Y^{-\frac{p}{t}}B)^{-1} \leq bA^{-1}A^{-*}\\
&\Rightarrow A(Q-X^{\frac{s}{t}}-B^*Y^{-\frac{p}{t}}B)^{-1}A^* \leq bI\\
&\Rightarrow F(X,Y)\leq bI.
\end{split}
\end{equation*}
Therefore $F$ maps $[aI,bI] \times [aI,bI]$ into $[aI,bI]$. Thus $F(aI,bI) \geq aI$ and $F(bI,aI)\leq bI$.\\
Let $X_1,X_2,Y \in [aI,bI]$ with $X_1 \leq X_2$. Then 
\begin{equation*}
\begin{split}
&~~~~X_1^{\frac{s}{t}} \leq X_2^{\frac{s}{t}}\\
&\Rightarrow Q-X_1^{\frac{s}{t}}-B^*Y^{-\frac{p}{t}}B \geq Q-X_2^{\frac{s}{t}}-B^*Y^{-\frac{p}{t}}B\\
&\Rightarrow A(Q-X_1^{\frac{s}{t}}-B^*Y^{-\frac{p}{t}}B)^{-1}A^* \leq A(Q-X_2^{\frac{s}{t}}-B^*Y^{-\frac{p}{t}}B)^{-1}A^*\\
&\Rightarrow F(X_1,Y) \leq F(X_2,Y).
\end{split}
\end{equation*}
Again let $Y_1,Y_2,X \in [aI,bI]$ with $Y_1 \leq Y_2$. Then
\begin{equation*}
\begin{split}
&~~~~Y_1^{-\frac{p}{t}} \geq Y_2^{-\frac{p}{t}}\\
&\Rightarrow Q-X^{\frac{s}{t}}- B^*Y_1^{-\frac{p}{t}}B \leq Q-X^{\frac{s}{t}}- B^*Y_2^{-\frac{p}{t}}B\\
&\Rightarrow A(Q-X^{\frac{s}{t}}- B^*Y_1^{-\frac{p}{t}}B)^{-1}A^* \geq A(Q-X^{\frac{s}{t}}- B^*Y_2^{-\frac{p}{t}}B)^{-1}A^*\\
&\Rightarrow F(X,Y_1) \geq F(X,Y_2).
\end{split}
\end{equation*}
Therefore $F$ has mixed monotone property.\\
Now let $X,Y,U,V \in [aI,bI]$ with $X \geq U$, $Y \leq V$, then \\
$
\| F(X,Y)-F(U,V) \| ~~~~~~~~~~~~~~~~~~~~~~~~~~~~~~~~~~\\
=\|A(Q-X^{\frac{s}{t}}-B^*Y^{-\frac{p}{t}}B)^{-1}A^* -A(Q-U^{\frac{s}{t}}-B^*V^{-\frac{p}{t}}B)^{-1}A^* \|\\
=\| A\{(Q-X^{\frac{s}{t}}-B^*Y^{-\frac{p}{t}}B)^{-1}-(Q-U^{\frac{s}{t}}-B^*V^{-\frac{p}{t}}B)^{-1}\}A^* \|\\
\leq \|A\|^2~\|(Q-X^{\frac{s}{t}}-B^*Y^{-\frac{p}{t}}B)^{-1}-(Q-U^{\frac{s}{t}}-B^*V^{-\frac{p}{t}}B)^{-1}\|\\
\leq \|A\|^2~\theta^{-2}~\|Q-X^{\frac{s}{t}}-B^*Y^{-\frac{p}{t}}B-Q+U^{\frac{s}{t}}+B^*V^{-\frac{p}{t}}B\|\\
\big(\textmd{Since, from (\ref{eq:18}), for }X,Y \in [aI,bI],\\~ Q-X^{\frac{s}{t}}-B^*Y^{-\frac{p}{t}}B \geq  b^{-1}A^*A > b^{-1}\lambda_n(A^*A)I=\theta I>0\big)\\
\leq \|A\|^2~\theta^{-2}\Big[\|X^{\frac{s}{t}}-U^{\frac{s}{t}}\|+\|B^*Y^{-\frac{p}{t}}B-B^*V^{-\frac{p}{t}}B\|\Big]\\
\leq \|A\|^2\theta^{-2}\Big[\frac{s}{t}a^{\frac{s}{t}-1}\|X-U\|+\|B\|^2~\|Y^{-\frac{p}{t}}-V^{-\frac{p}{t}}\|\Big]\\
\leq \|A\|^2~\theta^{-2}\Big[\frac{s}{t}a^{\frac{s}{t}-1}\|X-U\|+\|B\|^2 \frac{p}{t}a^{-\frac{p}{t}-1}\|Y-V\|\Big]\\
= \|A\|^2\theta^{-2}\frac{s}{t}a^{\frac{s}{t}-1}\|X-U\|+\|A\|^2~\theta^{-2}\|B\|^2 \frac{p}{t}a^{-\frac{p}{t}-1}\|Y-V\|\\
\leq \frac{\delta}{2}\big[\|X-U\|+\|Y-V\|\big]
$

where $\delta=2 \max\{\frac{s}{t}\|A\|^2~\theta^{-2}a^{\frac{s}{t}-1},~\frac{p}{t}\|A\|^2\|B\|^2\theta^{-2}a^{-\frac{p}{t}-1}\}$.
From conditions $(iii)$ and $(iv)$,  $0 \leq \delta <1$.  Also since $F$ is continuous in $[aI,bI]\times[aI,bI]$ and every pair $(X,Y) \in [aI,bI]\times[aI,bI]$ has an upper bound and a lower bound in it. So by using Theorem \ref{lemma 5} we conclude that there exists an $\bar{X} \in [aI,bI]$ such that $\bar{X}=F(\bar{X},\bar{X})$. Also this $\bar{X}$ is unique, which is a solution of equation (\ref{eq:17}). Subsequently,  equation (\ref{eq:1}) has a unique Hermitian positive definite solution $\bar{X}^\frac{1}{t}$ in $\big[a^\frac{1}{t}I,b^\frac{1}{t}I\big]$.\\
 $(III)$ immediately follows from the same theorem.

Now to prove $(II)$, we use the Schauder fixed point theorem. We define the mapping $G: [F(aI, bI),F(bI,aI)] \rightarrow [aI,bI]$ by 
\begin{equation*}
G(X)=F(X,X), \textmd{ for all } X \in [F(aI, bI),F(bI,aI)].
\end{equation*}
We claim that $G([F(aI, bI),F(bI,aI)]) \subset [F(aI, bI),F(bI,aI)]$.\\
 Let $X \in [F(aI, bI),F(bI,aI)]$, that is $F(aI,bI)\leq X \leq F(bI,aI)$. Then
\begin{equation}\label{21}
F(F(aI,bI), F(bI,aI)) \leq F(X,X)=G(X) \leq F(F(bI,aI),F(aI,bI))
\end{equation}
(Since $F$ is mixed monotone). \\
Also since $F(aI,bI) \geq aI$, $F(bI, aI) \leq bI$ we have
\begin{equation*}
F(F(bI,aI),F(aI,bI)) \leq F(bI,aI) \textmd{ and } F(F(aI,bI), F(bI,aI)) \geq F(aI,bI).
\end{equation*}
Then (\ref{21}) becomes $F(aI,bI) \leq G(X)\leq F(bI,aI)$. Thus, it proves our claim.\\
Now since $G$ maps the compact convex set $[F(aI, bI),F(bI,aI)]$ into itself and $G$ is continuous, it follows from Schauder fixed point theorem that $G$ has at least one fixed point $\bar{Y}$ (say) in this set. Thus $G(\bar{Y})=\bar{Y}$ $\Rightarrow F(\bar{Y},\bar{Y})=\bar{Y}$ and $[F(aI, bI),F(bI,aI)] \subset [aI,bI]$. But the fixed point of $F$ is unique in $[aI,bI]$. Thus $\bar{Y}=\bar{X} \in [F(aI, bI),F(bI,aI)]$. This proves $(II)$. 
\end{proof}

\begin{remark}
The obtained solution $\bar{X}$ in Theorem \ref{thm 7} is the minimal Hermitian positive definite solution of (\ref{eq:17}).
\end{remark}

\begin{proof}
Let $\bar{X_1}$ be any other Hermitian positive definite solution of (\ref{eq:17}) such that $\bar{X_1}\leq \bar{X}$. Then $\bar{X_1}\geq \lambda_n(AQ^{-1}A^*)I=aI$. Subsequently, $\bar{X_1}\in[aI,\bar{X}]\subset [aI,bI]$. But according to Theorem \ref{thm 7}, $\bar{X}$ is the only Hermitian positive definite solution in $[aI,bI]$. Therefore $\bar{X_1}=\bar{X}$. Thus $\bar{X}$ is the minimal Hermitian positive definite solution of (\ref{eq:17}). By similar explanation we also conclude that $\bar{X}^\frac{1}{t}$ is the minimal Hermitian positive definite solution of (\ref{eq:1}).	
\end{proof}

\begin{example}
		Consider the following nonlinear matrix equation 
	\begin{equation}\label{exeq:3}
	X^3+A^*X^{-4}A+B^*X^{-1}B=Q,
	\end{equation}
	 where  \hspace*{.1cm} $Q=\left( \begin{array}{ccc}
	7.5 & 0 & 1 \\
	0 & 7.5 & 1 \\
	1 & 1 & 8.5 \end{array} \right)$,  \hspace*{.1cm} $A= \left( \begin{array}{ccc} 2.11 & 0.01 & 0.01\\
	-0.05 & 1.98 & -0.18 \\
	0.1 & 0.19 & 2.38 
	\end{array} \right)$ and \\  \hspace*{3.1cm} $B=\left( \begin{array}{ccc} -0.09 & 0.01 & 0.01 \\
	-0.01 & -0.15 & -0.09 \\
	0.04 & 0.1 & -0.94 
	\end{array} \right)$. \\
	Then $t=4\geq s=3 \geq p=1$. Therefore in this case (\ref{exeq:3}) can be written as
	\begin{equation}\label{exeq:4}
	Y^\frac{3}{4}+A^*Y^{-1}A+B^*Y^{-\frac{1}{4}}B=Q,
	\end{equation} 
	where $Y=X^4\Rightarrow X=Y^\frac{1}{4}$. Let $b=1$.\\ Then $a=\lambda_n(AQ^{-1}A^*)=0.50754289893569<1=b$,  $\theta=b^{-1}\lambda_n(A^*A)=3.940180790866569$, $s\|A\|^2=17.269307701721161<26.207795027718277=\frac{1}{2}t \theta^2 a^{1-\frac{s}{t}}$ and $p\|B\|^2=0.90086767947051 < 1.52262869680707=sa^{\frac{p+s}{t}}$.\\
	Also in that case $b^{-1}A^*A+b^{\frac{s}{t}}I+a^{-\frac{p}{t}}B^*B  \leq Q$ holds. Thus all the hypothesis of Theorem \ref{thm 7} are satisfied with $b=1$. After $20$ iterations, we get the unique solution of equation (\ref{exeq:4}) in $[aI, bI]$ as\\ $\bar{X}= X_{20}=\\Y_{20}$=$\left( \begin{array}{ccc} 
	0.678793416023482 & 0.017053803392642 & -0.094857343070291 \\
	0.017053803392642 & 0.622769611868454 & -0.138376527663483 \\
	-0.094857343070291 & -0.138376527663483 & 0.872777116839001
	\end{array} \right)$, \\
	with error $e=5.22834858061468 \times 10^{-16}$.\\
	 Considering the sequence $X_n$, $Y_n$ with $X_0=0.50754289893569 I$ and $Y_0=I$, where $X_{n+1}=F(X_n,Y_n)$ and $Y_{n+1}=F(Y_n,X_n)$ we approach the solution. For each iteration $n$, we get the error as $e'=\|X_n-X_{n+1}\|$ (curve 1), $f=\|Y_n-Y_{n+1}\|$ (curve 2) and $e=\max\{e',f\}$. 
	 Therefore, equation (\ref{exeq:3}) has a unique Hermitian positive definite solution in $[a^{\frac{1}{4}}I, b^{\frac{1}{4}}I]$ as\\
	 
	 $\left( \begin{array}{ccc} 
	 0.906231149966594 & 0.003723228318032 & -0.028702574652700 \\
	 0.003723228318032 & 0.884927869436603 & -0.043501905340609 \\
	 -0.028702574652700 & -0.043501905340609 & 0.962538505271393
	 \end{array} \right)$.  \\
		 
	 \begin{figure}
	 	\begin{minipage}{0.7\textwidth}
	 		\includegraphics[width=1.0\textwidth]{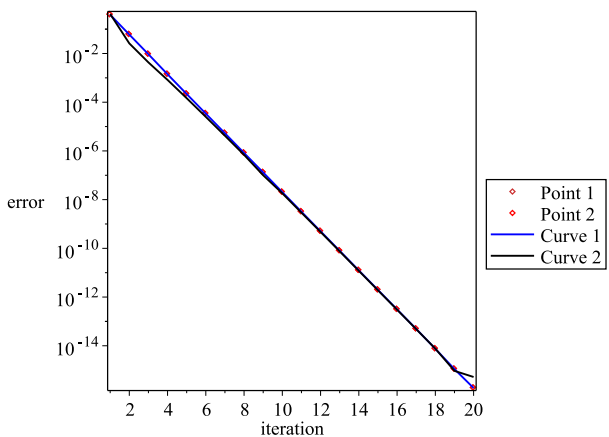}\\
	 		\caption{Convergence history of (\ref{exeq:4})}
	 		\label{Fig.2}
	 		\end{minipage}
	 \end{figure}
\end{example}

\section*{\textbf{Acknowledgements}}
Samik Pakhira gratefully acknowledges the financial support provided by CSIR, Govt. of India.
\pagebreak

\bibliographystyle{amsplain}

\end{document}